%% file: main.tex
\documentclass[final,onefignum,onetabnum]{siamart190516}

\usepackage[utf8]{inputenc}
\usepackage{amsmath}
\usepackage{amsfonts}
\usepackage{amssymb}

\usepackage{pgf,tikz}
\usetikzlibrary{automata, positioning, arrows}

\usepackage{algorithm,algpseudocode}
\usepackage{listings}
\usepackage[caption=false]{subfig}
\usepackage{multirow}
\usepackage{todonotes}

\newsiamremark{remark}{Remark}

\DeclareMathOperator*{\argmin}{argmin}

\DeclareMathOperator{\range}{range}
\newcommand{\intr}{\text{int}}
\newcommand{\extr}{\text{ext}}
\newcommand{\tend}{t_{\text{end}}}
\newcommand{\tprev}{\lfloor t \rfloor}
\newcommand{\tnext}{\lceil t \rceil}
\newcommand{\tsync}{t_{\text{sync}}}

\newcommand{\pq}{\mathcal{Q}}
\newcommand{\update}{\mathcal{U}}
\newcommand{\pushflux}{\mathcal{PF}}
\newcommand{\nsbmsh}{n_{\text{sbmsh}}}

\input{images/results_macros}

\definecolor{emerald}{rgb}{0.3125, 0.78125, 0.46875}

\definecolor{tuscany}{rgb}{0.988235, 0.81961, 0.1922}

\definecolor{cambridge-blue}{rgb}{0.637, 0.754, 0.676}

\title{Adaptive Total Variation Stable Local Timestepping for Conservation Laws}
\author{Maximilian Bremer, John Bachan, Cy Chan, Clint Dawson}

\begin{document}
\maketitle

\begin{abstract}
    This paper proposes a first-order total variation diminishing (TVD) treatment for coarsening and refining of local timestep size in response to dynamic local variations in wave speeds for nonlinear conservation laws. The algorithm is accompanied with a proof of formal correctness showing that given a sufficiently small minimum timestep the algorithm will produce TVD solution for nonlinear scalar conservation laws.
A key feature of the algorithm is its formulation as a discrete event simulation, which allows for easy and efficient parallelization using existing software.
Numerical results demonstrate the stability and adaptivity of the method for the shallow water equations.
We also introduce a performance model to load balance and explain the observed performance gains.
Performance results are presented for a single node on Stampede2’s Skylake partition using an optimistic parallel discrete event simulator.
Results show the proposed algorithm recovering 59\%-77\% of the theoretically achievable speed-up with the discrepancies being attributed to the cost of computing the CFL condition and load imbalance.
\end{abstract}

\input{sections/introduction} 

\input{sections/previous_work} 

\input{sections/first_order_method} 

\input{sections/code_snippets}
\input{sections/algorithm} 
\input{sections/implementation_details}

\input{sections/results} 

\input{sections/conclusion} 
{\section*{Acknowledgments}

This work was supported by the NSF under Grants 1339801 and 1854986, the U.S. DOE through the Computational Science Graduate Fellowship, Grant DE-FG02-97ER25308, and the University of Texas at Austin's Harrington Fellowship.
Additionally, the authors would like to acknowledge the Texas Advanced Computing Center (TACC) at the University of Texas at Austin for providing the HPC resources that enabled this research. The presented results were obtained through XSEDE allocation TG-DMS080016N.

This manuscript has been authored by an author at Lawrence Berkeley National Laboratory under Contract No. DE-AC02-05CH11231 with the U.S. DOE. The U.S. Government retains, and the publisher, by accepting the article for publication, acknowledges, that the U.S. Government retains a non-exclusive, paid-up, irrevocable, world-wide license to publish or reproduce the published form of this manuscript, or allow others to do so, for U.S. Government purposes.
}
\bibliographystyle{siamplain}
\bibliography{BremerBib/Timestepping,BremerBib/HPC,BremerBib/DG,BremerBib/SWE,BremerBib/DES,BremerBib/LoadBalancing}

\end{document}

%% file: images/results_macros.tex
\newcommand{\SobsNoCFL}{0.96}

\newcommand{\maxSpeedupErr}{6.0}
\newcommand{\polynomialImb}{1.02}
\newcommand{\polynomialImbwRb}{1.13}
\newcommand{\imbCGdtRb}{0.02}

\newcommand{\minSObsOverWork}{59}
\newcommand{\maxSObsOverWork}{77}

%% file: sections/introduction.tex
\section{Introduction}
\label{sec:intro}

The total variation stability results for scalar conservation laws form the basis of the robustness which has led to the popularity of finite volume and discontinuous Galerkin finite element methods. Roughly, the Courant-Friedrichs-Lax (CFL) condition stipulates that under a restriction on the timestep as a function of wave speed $|\Lambda|$ and mesh size $\Delta x$,
\begin{equation}
\Delta t \le \frac{C_{CFL} \Delta x}{|\Lambda|},
\label{eq:vanilla-cfl}
\end{equation}
the numerical solution will weakly converge to a weak solution~\cite{LeVeque1992}. However, for large scale simulations, it is computationally too expensive to check that~\eqref{eq:vanilla-cfl} is enforced due to high communication overhead. In practice, timesteps are set to sufficiently small values.

Enforcing a global CFL condition or using a sufficiently small global timestep tends to be overly conservative for large portions of the mesh. For example, in the case of hurricane storm surge simulations, meshes are used with length scales ranging from $\mathcal{O}(10\,\mathrm{m})$ to $\mathcal{O}(1\,\mathrm{km})$~\cite{Dawson2011}. Local timestepping relaxes the global CFL constraint~\eqref{eq:vanilla-cfl} by allowing different regions of the mesh to advance with different local timesteps. Similar to synchronous timestepping, existing methods are unable to account for significant temporal variations in the wave speed $|\Lambda|$.

\input{images/shock_motivation}

As a motivating example, consider (inviscid) Burgers' equation ($f(u) = u^2/2$), with initial conditions,
\begin{equation}
u_0(x) = \begin{cases}
1 & u < 0,\\
0 & u > 0.
\end{cases}
\label{eq:shock-ics-mot}
\end{equation}
The solution to this problem is a shockwave moving with speed $1/2$ to the right, i.e. $u(x,t) = u_0(x - t/2)$. The local wavespeed is given by $|\Lambda|(x) = |u|(x,t)$. For clarity, we have depicted the schematic in Figure~\ref{fig:shock-motivation}. Consider the cell $j$ downstream of the shock. Were one to naively evaluate the CFL condition~\eqref{eq:vanilla-cfl}, they would incorrectly determine that cell $j$ is able to step arbitrarily far. In doing so, as the shock arrives at cell $j$, it would be unable to pass through the cell, and mass would accumulate at the interface. Such a scheme is unable to converge.

Current theoretical results require knowing the entire temporal history of the flux, before being able to determine that a timestep will be stable. Therefore, one cannot look only at neighboring values, but must consider all flux values between the last update and the next update. While existing works can assess whether or not a timestep will be stable, they fail to describe an algorithm that will advance the system in a manner that guarantees stability. This work addresses this issue.


The main result of this paper is a novel adaptive local timestepping algorithm, which is provably total variation stable. The algorithm recasts local timestepping as a discrete event simulation, which allows cells to dynamically coarsen and refine their timesteps. A proof of correctness is supplied to verify that under a sufficiently small minimum timestep, the algorithm will stably advance the system to an arbitrary final time, $\tend$.  We also introduce a performance model for the adaptive, locally time-stepped method, which is used to both load balance the execution and explain the observed performance gains relative to theoretical speedup bounds.   While the presented algorithm is first order in time, this work makes two fundamental contributions, which will enable local timestepping to become feasible for nonlinear hyperbolic problems in a massively parallel computing environment:
\begin{enumerate}
    \item {\em The application of loop invariants in the proof of correctness}: Loop invariants are a proof technique that guarantee a program is formally correct, and have been used with great success in the linear algebra community~\cite{Bientinesi2011}. Given the highly asynchronous context in which our timestepping method is executed, it is extremely difficult to debug such a program. By using these formal correctness techniques, we have machinery which ensures that the algorithm achieves desired numerical properties. This not only strengthens the confidence in the results produced by the algorithm, but provides a systematic way to manage the complexity associated with higher order timestepping methods. We expect this proof technique to be indispensable in the extension of this timestepping scheme to higher orders and multiple dimensions.
    \item {\em A discrete event formulation which removes artificial synchronizations}: Many current timestepping formulations consider only two timestepping groups. Extending these formulations to more timestepping groups is done in a recursive fashion. However, parallelizing these methods then requires a synchronization between each fine timestep to allow for timestepping adapativity. 
    Amdahl's law severely restricts the scalability of any such formulation. By using a discrete event simulation, we consider arbitrarily many timestepping groups. Furthermore, we can leverage the extensive work done for parallel discrete event simulation to arrive at an efficient parallel implementation. Parallel performance is demonstrated for a single node on Stampede2's Skylake architecture using Devastator, one such parallel discrete event simulator~\cite{Chan2018}.
\end{enumerate}

The remainder of this paper is structured as follows. In Section \ref{sec:prev}, we discuss the current state of the art. Section \ref{sec:method} generalizes existing local timestepping results to account for arbitrary local timestepping. Section \ref{sec:alg} presents the timestepping algorithm, which we accompany with a proof of correctness. Section~\ref{sec:implementation} introduces the Devastator runtime and performance optimizations for the local timestepping algorithm. Finally, one-dimensional numerical results for the shallow water equations are presented in Section~\ref{sec:results}.

%% file: images/shock_motivation.tex
\begin{figure}
\centering
  \begin{tikzpicture}
    \draw[<->,thick] (-4.8,0)--(4.8,0);
    \draw (-4.4,1)-- (-0.944,1) node[midway,above] {$u_0=1$};
    \draw (-0.944,1)--(-0.944,0);
    \draw[->,thick] (-1.294,0.5)--(-0.594,0.5);
    \draw[|-|] (1,0)--(1.6,0) node[midway,below] {$j$};
  \end{tikzpicture}
  \caption{Riemann problem for Burgers' equation with initial conditions, Equation~\eqref{eq:shock-ics-mot}.}
  \label{fig:shock-motivation}
\end{figure}
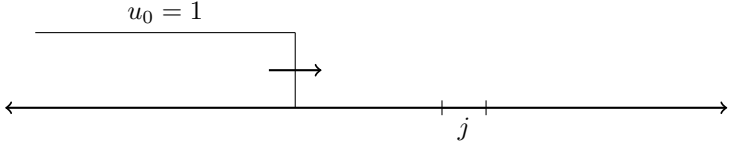

%% file: sections/previous_work.tex
\section{Previous Work}
\label{sec:prev}

Relaxing the CFL condition for conservation laws by allowing different regions of the mesh to advance with different timesteps has been the subject of numerous studies. Osher et al. presented a first-order total variation diminishing timestepping scheme in~\cite{Osher1983}. Dawson and Kirby developed a second order method, which reduces to first order at the interface between elements stepping with different timesteps~\cite{Dawson2000}. Kirby generalized this approach to high resolution methods in~\cite{Kirby2003}. Constantinescu developed a second order total variation stable method for conservation laws~\cite{Constantinescu2007}. Mass conservative but not provably total variation stable third order methods were presented in the work of Schlegel et al~\cite{Schlegel2009}.

Other methods have relied on high-order interpolation strategies to approximate coupled terms with larger timesteps~\cite{Krivodonova2010, Gupta2016, Hoang2019}. Another set of local timestepping algorithms for conservation laws rely on single step methods with high-order space-time representations~\cite{Loercher2007,Dumbser2007,Taube2009}. These methods are conservative, but not provably total variation stable.

The main thrust of this paper is a reformulation of the algorithm of Osher et al~\cite{Osher1983} to improve parallel performance. Previous work parallelizing local timestepping approaches has had mixed results. For linear conservation laws, task-based programming models have been extremely effective, with local timestepping ADER-DG methods scaling up to full system runs~\cite{Breuer2016,Uphoff2017}. 
Successful task-based implementations for linear conservation laws fundamentally rely on timestepping groups being statically determined and thus are unsuitable for timestepping adaptivity.
Ignoring dynamic changes in timestepping groups results in CFL violations and associated instabilities, e.g.~\cite{Trahan2012,Gnedin2018}. 
Other parallelization attempts for conservation laws have been made by recursively updating finer and finer timestepping groups. These methods would allow for adaptive timestepping by synchronizing after each timestep to allow for cells to change their timestepping groups. However, load balancing these methods requires balancing the work in each timestepping group across all ranks~\cite{Rietmann2015,Seny2014}. The scalability of these approaches is limited by the amount of work in the timestepping group with the smallest timestep.

Other parallelization approaches have treated the adaptive local timestepping problem as a discrete event simulation~\cite{Omelchenko2006, Unfer2007, Omelchenko2012, Omelchenko2014,Stone2017,Shao2019}. While these approaches achieve tremendous speed-ups, these methods typically provide only heuristics and examples as proofs of robustness. To the authors' knowledge, no method has presented a provably total-variation stable adaptive timestepping scheme.

%% file: sections/first_order_method.tex
\section{Theoretical Results for Scalar Conservation Laws}
\label{sec:method}
In this paper, we are concerned with solving problems of the following form: Find $u \in L^{\infty}((0,\tend); \mathbb{T})$ such that
\begin{equation}
\begin{cases}
\partial_t u + \partial_x f(u) = 0&\\
u(x,0) = u_0(x),&
\end{cases}
\label{eq:ivp}
\end{equation}
where $f$ is a Lipschitz flux, $u_0(x) \in L^{\infty}(\mathbb{T})\cap BV(\mathbb{T})$, and $\mathbb{T}$ is the unit torus.  Consider a finite partitioning of the torus, $\Omega^h = \cup_{j=0}^{n_{el}} (x_j,x_{j+1})$ as defined by the strictly increasing sequence $\{x_j\} \subset [0,1]$. Let $U_j$ denote the average value on the interval $(x_j, x_{j+1})$, and let $\Delta$ denote the forward difference operator, e.g. $\Delta U_k = U_{k+1} - U_k$.
We specify an approximation to the conservation law as a set of pairs,
\begin{equation*}
\mathcal{E} = \left\{ (t_j^n, U_j^n)\,:\, 0 \le j \le n_{el},\,n\in\mathbb{R}_+ \right\},
\end{equation*}
where $U_j^n$ is the average value over the cell $(x_j,x_{j+1})$ at time $t_j^n$. We refer to this collection of space-time points as the {\em event trace}. We define the timestamps $\mathcal{T}_j$ on cell $j$, as
\begin{equation*}
\mathcal{T}_j = \left\{ t_j^* \, : \, \exists\, U_j^* \, \text{s.t.}\, (t_j^*,U_j^*) \in \mathcal{E} \right\}.
\end{equation*}
We define for cell $j$, the nearest previous timestep for time $t$ as
\begin{equation*}
\lfloor t \rfloor_j = \max \{ \tau \in \mathcal{T}_j \, : \, \tau \le t \}.
\end{equation*}
We similarly define the next timestep at time $t$ as
\begin{equation*}
\lceil t \rceil_j = \min \{ \tau \in \mathcal{T}_j\, : \, \tau > t \}.
\end{equation*}
To achieve the desired theoretical results, we impose two constraints on the event traces. Firstly, we assume that there are only a finite number of events. Secondly, we need to restrict when timestamps are able to step relative to one another. 

We call the set of timestamps $\cup_j \mathcal{T}_j$ \emph{locally ordered} if for each cell $j$ there exists a sequence of pairwise synchronization times $\{s^{\mu}_{j,j+1}\}_{\mu=0}^{n_s} \subset \mathcal{T}_j \cap \mathcal{T}_{j+1}$ such that for all consecutive synchronization times $s^{\mu}_{j,j+1}$ and $s^{\mu+1}_{j,j+1}$,
\begin{equation*}
\big( \mathcal{T}_j \cup \mathcal{T}_{j+1} \big) \cap (s^{\mu}_{j,j+1}, s^{\mu+1}_{j,j+1}) = \emptyset.
\end{equation*}
That is to say that one cell must step strictly faster than its neighbor between  synchronization times. To illustrate this definition, we've drawn a sketch of potential local timestepping solutions in Figure \ref{fig:sol-cartoon}.
\input{images/event_trace_cartoon}
Additionally, we define the value of an approximation at any point in time and space as the most recent average value of the cell containing the given point, i.e. 
\begin{equation*}
U(t,x) = \begin{cases}
U_j^n & \text{if } x \in (x_j, x_{j+1}) \quad \text{and} \quad t = t_{j^n},\\
U(\tprev_j, x) & \text{where } x \in (x_j, x_{j+1}).
\end{cases}
\end{equation*}
We also introduce $U_j(t)$ as the value of the solution $U(t,x)$ inside cell $j$ at time $t$. To approximate the flux exchange at the boundary, we define a numerical flux $\hat{F}(\cdot, \cdot)$. Additionally, we require that the numerical flux be:
\begin{enumerate}
\item {\em Consistent}, i.e. $\hat{F}(u^*,u^*) = f(u^*)$,
\item {\em Monotone}, i.e. $\hat{F}$ is increasing in the first argument, and decreasing in the second argument,
\item {\em Lipschitz continuous} in both arguments.
\end{enumerate}
We now define the Euler approximation to the scalar conservation law as follows.
\begin{definition}[Forward Euler Approximation]
An event trace $\mathcal{E}$ is an Euler approximation to the scalar conservation law if
\begin{equation}
U_j^{n+1} = U_{j}^n + \frac{1}{\Delta x_j}\int_{t_j^n}^{t_j^{n+1}} \big[ \hat{F}(U_{j-1}(\tau), U_j^n) - \hat{F}(U_j^n,U_{j+1}(\tau)) \big] \,\mathrm{d}\tau,
\label{eq:update}
\end{equation}
for all cells $j$, and between all updates $t_j^n$ and $t_j^{n+1}$, where $\hat{F}$ is a numerical flux.
\end{definition}

The remainder of this section assumes the existence of forward Euler approximations. Under this assumption, we present proofs that under CFL-like constraints, a maximum principle and total variation stability can be obtained. 
Following the spirit of the analysis proposed in~\cite{Harten1983}, define
\begin{equation*}
C_j = \frac{ \hat{F}(U_{j-1},U_{j+1}) - \hat{F}(U_{j},U_{j+1})}{\Delta x_j( U_j - U_{j-1})} \text{ and } D_j = \frac{  \hat{F}(U_{j-1},U_{j}) - \hat{F}(U_{j-1},U_{j+1})}{\Delta x_j (U_{j+1} - U_{j})}.
\end{equation*}
We remark that due to the Lipschitz continuity of the numerical flux both $C_j$ and $D_j$ are bounded, and due to the monotonicity of the numerical flux $C_j \le 0 \le D_j$ for all $U_{j-1},U_j,U_{j+1}\in \mathbb{R}$. We reformulate the update rule \eqref{eq:update} as
\begin{equation}
U_j^{n+1} = U_j^n + \int_{t_j^n}^{t^{n+1}_j} \big[ C_j(\tau) \Delta U_{j-1}(\tau) - D_j(\tau) \Delta U_j(\tau) \big] \,\mathrm{d} \tau.
\label{eq:update2}
\end{equation}
Note that this representation is slightly different than the analyses presented in~\cite{Harten1983,Osher1983,Kirby2003}. We do not multiply our $C_j$ and $D_j$ coefficients by $\Delta t$.

\begin{theorem}[Maximum Principle]
Consider an event trace $\mathcal{E}$ satisfying the forward Euler approximation~\eqref{eq:update}. If given any two events $(t_j^n,\,U_j^n)$ and $(t_j^{n+1},\,U_j^{n+1})$,
\begin{equation}
1 + \int_{t_j^n}^{t_j^{n+1}} \big[ C_j(\tau) - D_j(\tau) \big] \,\mathrm{d}\tau \ge 0,
\label{eq:cfl-max}
\end{equation}
then
\begin{equation*}
|U^n_j| \le \sup_j |U^0_j|.
\end{equation*}
\label{thm:max}
\end{theorem}

\begin{proof}
Taking the absolute value of the update criterion~\eqref{eq:update2} and applying the triangle inequality,
\begin{align*}
\left| U^{n+1}_j \right| & \le
\left| U_j^n + \int_{t^n_j}^{t^{n+1}_j} \big[ C_j(\tau) U_j^n - D_j(\tau)  U_j^n \big] \mathrm{d} \tau \right| \\
& \quad + \left| \int_{t^n_j}^{t^{n+1}_j} C_j(\tau) U_{j-1}(\tau) \mathrm{d} \tau \right| + \left| \int_{t^n_j}^{t^{n+1}_j} D_j(\tau) U_{j+1}(\tau) \mathrm{d} \tau \right|.
\end{align*}
Using the CFL condition~\eqref{eq:cfl-max}, and $C_j(\tau) \le 0 \le D_j(\tau)$ for all $\tau$,
\begin{align*}
\left| U^{n+1}_j \right| & \le \left( 1+ \int_{t^n_j}^{t^{n+1}_j} C_j(\tau)  + D_j(\tau) \, \mathrm{d} \tau \right) |U_j^n|\\
& \quad - \int_{t^n_j}^{t^{n+1}_j} C_j(\tau) |U_{j-1}(\tau)| \,\mathrm{d}\tau + \int_{t^n_j}^{t^{n+1}_j} D_j(\tau) |U_{j+1}(\tau)| \,\mathrm{d}\tau.
\end{align*}
Let $U^* = \sup_{\tau \in [t_j^n,t_j^{n+1})} \{ |U_j^n|, |U_{j-1}(\tau)|, |U_{j+1}(\tau)| \}$, then \
\begin{align*}
|U^{n+1}_j| & \le \left( 1 +\int_{t_j^n}^{t_j^{n+1}} \big[ C_j(\tau) - D_j(\tau) \big]\mathrm{d} \tau \right) U^* \\
& - \left(  \int_{t_j^n}^{t_j^{n+1}} C_j(\tau) \mathrm{d} \tau \right) U^* + \left(  \int_{t_j^n}^{t_j^{n+1}} D_j(\tau) \mathrm{d} \tau \right) U^* \le U^*.
\end{align*}
Lastly, we define the minimum time between events as
\begin{equation*}
\varepsilon = \inf_{j,k,n} \left\{ t_j^{n} - t \, : \, t \in \mathcal{T}_k \land t < t_j^n \right\}.
\end{equation*}
Since we  assume a finite number of events, $\varepsilon>0$. Letting $\mathcal{U}(t)$ be defined as the largest value up till time $t$, i.e.
\begin{equation*}
\mathcal{U}(t) = \max_{j,\,\tau \le t} |U_j(\tau)|.
\end{equation*}
By the definition of $\varepsilon$, $U^* \le \mathcal{U}(t^{n+1}_j - \varepsilon/2)$.
Finally, we claim $\mathcal{U}((n+1)\varepsilon/2) \le \mathcal{U}(n\varepsilon/2)$. Pick any event at $t_i^m$ such that $n\varepsilon/2 < t_i^m \le (n+1) \varepsilon/2$. By the above analysis, 
\begin{equation*}
|U_i^m| \le \mathcal{U}(t_m^i - \varepsilon/2) \le \mathcal{U}(n\varepsilon/2).
\end{equation*}
Taking the maximum over all events occurring between $n\varepsilon/2$ and $(n+1)\varepsilon/2$, $\mathcal{U}((n+1) \varepsilon/2) \le \mathcal{U}(n \varepsilon)$. Arguing inductively, for all events $(t_j^n,\,U_j^n)$, $|U_j^n| \le \mathcal{U}(0) = \sup_j |U_j^0|$.
\end{proof}


In addition to a maximum principle, we can similarly formulate a CFL condition to ensure that the event trace is total variation diminishing.
\begin{theorem}
\label{thm:tvd-stab}
A locally ordered forward Euler solution $\mathcal{E}$ subject to the following CFL constraint: for all cells $j$ and for all times between consecutive synchronization times, $t \in (s_{j,j+1}^{\mu}, s_{j,j+1}^{\mu+1})$,
\begin{equation}
1 + (\lceil t \rceil_{j+1} - s_{j,j+1}^{\mu})C_{j+1}(t) - (\lceil t \rceil_j - s_{j,j+1}^{\mu})D_{j}(t) \ge 0,
\label{eq:cfl-tvd}
\end{equation}
 is total variation diminishing (TVD), i.e.
\begin{equation*}
TV( U(t) ) = \sum_j |U_{j+1}(t) - U_j(t)| \le TV(U(0)).
\end{equation*}
\end{theorem}

Since the solution is locally ordered, either cell $j$ or $j+1$ must substep relative to the other cell. Between consecutive synchronization times, the scheme reduces to a traditional two step local timestepping scheme, allowing us to leverage existing work.

\begin{lemma}
\label{lem1}
Given a locally ordered forward Euler solution $\mathcal{E}$, consider cells $j$ and $j+1$ and consecutive synchronization times $s_{j,j+1}^{\mu}$ and $s_{j,j+1}^{\mu+1}$. If for all times between the synchronization times, $t \in (s_{j,j+1}^{\mu}, s_{j,j+1}^{\mu+1})$, Equation~\eqref{eq:cfl-tvd} holds, then
\begin{equation*}
    |\Delta U_j(s_{j,j+1}^{\mu+1})| \le |\Delta U_j(s_{j,j+1}^{\mu})| + \Delta \int_{s_{j,j+1}^{\mu}}^{s_{j,j+1}^{\mu+1}} C_j |\Delta U_{j-1}|(\tau) + D_j|\Delta U_j|(\tau)\,\mathrm{d}\tau.
\end{equation*}
\end{lemma}
Although the CFL condition \eqref{eq:cfl-tvd} is less stringent than the one presented in Kirby~\cite{Kirby2003}, the proof for Lemma~\ref{lem1} follows identically to the proof for Proposition 4.1 in~\cite{Kirby2003}.

The proof of Theorem \ref{thm:tvd-stab} then follows. Since the event trace is locally ordered, we apply Lemma~\ref{lem1} to each synchronization time of cells $j$ and $j+1$ to get
\begin{equation*}
    |\Delta U_j(t)| \le |\Delta U_j(0)| + \Delta \int_0^t C_j|\Delta U_{j-1}|(\tau) + D_j|\Delta U_j|(\tau) \,\mathrm{d} \tau.
\end{equation*}
Summing across all cells $j$ yields the result.

%% file: images/event_trace_cartoon.tex
\begin{figure}
\centering
\subfloat[Straddling]{
\begin{tikzpicture}
    \draw[<->,thick] (-0.3,0)--(3.3,0) node[midway,below] {Space};
    \draw[<->,thick] (0,-0.3)--(0,3.3) node[midway,left] {Time};
    \draw[-] (0.75,0)--(0.75,3);
    \draw[-] (1.75,0)--(1.75,3);
    \draw[-] (2.75,0)--(2.75,3);
    \draw[-] (0.75, 0.45)--(1.75, 0.45);
    \draw[-] (0.75, 0.55)--(1.75, 0.55);
    \draw[-] (0.75, 1.45)--(1.75, 1.45);
    \draw[-] (0.75, 2.25)--(1.75, 2.25);
    \draw[-] (1.75, 1)--(2.75, 1);
    \draw[-] (1.75, 2.55)--(2.75, 2.55);
\end{tikzpicture}
}\hspace{1in}
\subfloat[Locally Ordered]{
\begin{tikzpicture}
    \draw[<->,thick] (-0.3,0)--(3.3,0) node[midway,below] {Space};
    \draw[<->,thick] (0,-0.3)--(0,3.3) node[midway,left] {Time};
    \draw[-] (0.75,0)--(0.75,3);
    \draw[-] (1.75,0)--(1.75,3);
    \draw[-] (2.75,0)--(2.75,3);
    \draw[-] (0.75, 0.25)--(1.75, 0.25);
    \draw[-] (0.75, 0.5)--(1.75, 0.5);
    \draw[-] (0.75, 2.25)--(1.75, 2.25);
    \draw[-] (1.75, 0.5)--(2.75, 0.5);
    \draw[-] (1.75, 1.55)--(2.75, 1.55);
    \draw[-] (1.75, 2)--(2.75, 2);
    \draw[-] (1.75, 2.25)--(2.75, 2.25);
\end{tikzpicture}
}
\caption{Comparison of two event traces.}
\label{fig:sol-cartoon}
\end{figure}
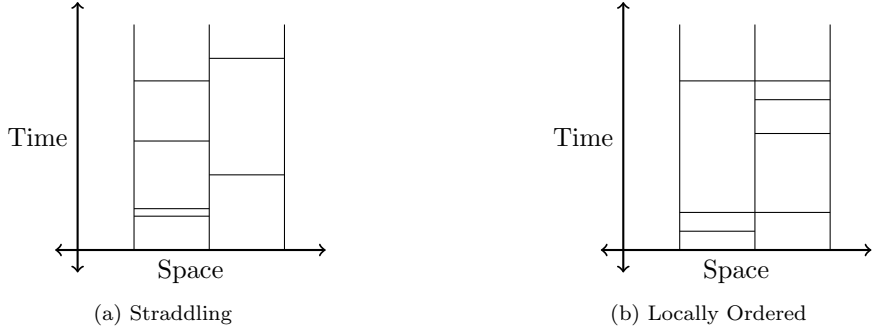

%% file: sections/code_snippets.tex

\newsavebox{\submeshlisting}
\begin{lrbox}{\submeshlisting}
\begin{lstlisting}[mathescape=true,basicstyle=\scriptsize]
class Submesh
  $id \in \mathbb{Z}$
  $\tprev,\,\tnext\in \mathbb{R}$
  $u,\,\Delta x \in \mathbb{R}^{n_{el}}$
  $\ell,\,r \in$ Interface
\end{lstlisting}
\end{lrbox}

\newsavebox{\interfacelisting}
\begin{lrbox}{\interfacelisting}
\begin{lstlisting}[mathescape=true,basicstyle=\scriptsize]
class Interface
  $id \in \mathbb{Z}$
  $\tprev,\,\tsync \in \mathbb{R}$
  $u,\,\Sigma \hat{F},\,\Delta x \in \mathbb{R}$
  $K^{int},\,K^{ext} \in \mathbb{R}$
\end{lstlisting}
\end{lrbox}

%% file: sections/algorithm.tex
\section{An Adaptive Local Timestepping Algorithm}
\label{sec:alg}
In this section, we present an algorithm that generates an event trace satisfying the hypotheses of Theorem \ref{thm:tvd-stab}. To allow for an efficient implementation, we couch the algorithm in the language of discrete event simulations. The main result of this section will prove that with a sufficiently small minimum timestep, our algorithm will generate a total variation stable event trace.

\subsection{Discrete Event Simulation}
Discrete event simulation as a tool has been used extensively for the simulation of numerous phenomena~\cite{Fujimoto1990}. Fundamentally, the simulation is represented as a set of actors, and a set of timestamped events, denoted $\pq$.  The discrete event simulator guarantees that events will be executed in order of their timestamps. When executed, an event may schedule more events at later timestamps.
At a high level, the actors used for the approximation of solutions to conservation laws will consist of submeshes. For this section, we require that at least two cells be assigned to each submesh. 
However, achieving performance requires balancing the overhead of the simulator against the amount of useful work being done during each task~\cite{Bremer2019}. In practice, submeshes should contain significantly more cells. Each submesh can schedule one of two events: (1) {\em Update} ($\mathcal{U}$) the submesh to the time at which the update is scheduled, and (2) {\em Push Flux} ($\mathcal{PF}$), wherein the submesh sends relevant metadata to the neighboring cell to allow advancing along the shared boundary. We quantize our time to be an integer multiple of a smallest timestep, $\Delta t_{\min}$. In this section, we use $t,\,\tprev,\,\tnext$ as variable names and $\tau$ is used to specify timestamps in the discrete event simulation.

\begin{figure}
\centering
\subfloat[Submesh class\label{class:submesh}]{\usebox{\submeshlisting}}\hspace{2in}
\subfloat[Interface class\label{class:interface}]{\usebox{\interfacelisting}}
\caption{Local timestepping data structures}
\end{figure}

The actor states are described in Figure~\ref{class:submesh}. The actor's members correspond to
\begin{itemize}
\item $id$: The submesh id,
\item $\tprev$: The last time at which the submesh was updated,
\item $\tnext$: The next time at which the submesh is scheduled to be updated,
\item $u$: The average density on each cell,
\item $\Delta x$: The sizes of the cells for the submesh,
\item $\ell,\,r$: Representations of the states to the left and right of the submesh.
\end{itemize}
The interface class in Figure~\ref{class:interface} describes neighboring cell metadata and has members describing: 
\begin{itemize}
\item $id$: The id of the neighboring submesh,
\item $\tprev$: The last time at which the neighbor was updated,
\item $\tsync$: The last time at which the neighbor and submesh both updated,
\item $u$: The neighbor's density at time $\tprev$,
\item $\Sigma \hat{F}$: The integral of the numerical flux between the submesh and its neighbor from the submesh's $\tprev$ to the interface's $\tprev$,
\item $\Delta x$: The size of the neighboring cell,
\item $K^*$: The largest Lipschitz constant of the numerical flux between time $\tsync$ to $\tprev$ on the internal interface ($K^{\intr}$) or external interface cell ($K^{\extr}$).
\end{itemize}
We note that the Lipschitz constant $K$ is used to bound $C_j$ and $D_j$ terms simultaneously. For commonly used fluxes like the Godunov flux or Lax-Friedrichs flux, $K$ is set to $|\Lambda|/\Delta x$ yielding the commonly seen version of the CFL condition \eqref{eq:vanilla-cfl}.
Member fields are denoted using a period, e.g. $S.\ell.\tprev$ denotes the previous update of the left interface of submesh $S$.
Before we define our two main events, {\sc update} and {\sc push\_flux}, we define some helper functions:
\begin{itemize}
\item {\sc advance(t, submesh)}: Advance the submesh one timestep from submesh.$\tprev$ to $t$ according to the update rule~\eqref{eq:update}. Note that this will also cause corresponding updates in boundary terms such as $K^*$ and $\Sigma \hat{F}$.
\item {\sc compute\_t\_next\_bdry(submesh, neigh)}: Compute the timestep for interfaces which depend on neigh.$u$. This evaluates two CFL conditions at: (1) the external interface between the submesh and its neighbor, which is synchronized at neigh.$\tsync$ and (2) the internal interface nearest to the external interface, which is synchronized at submesh.$\tprev$.
\item {\sc compute\_t\_next(submesh, t)}: Compute the timestep for the submesh. This is taken to be the minimum of the values returned by {\sc compute\_t\_next\_bdry} for both neighbors and the largest allowable internal timestep.
\item {\sc make\_msg(submesh, neigh)}: Generate a buffer with the required information to update the neighbor's corresponding interface.
\item {\sc accumulate(t, submesh, neigh)}: Update neigh.$\Sigma \hat{F}$ to integrate from the previous integration point neigh.$\tprev$ to time $t$. 
\item {\sc update\_K\_bdry(submesh, neigh)}: Update Lipschitz constants $K^*$ according to the new updated values at the boundary.
\end{itemize}
The helper functions serve as an API between the application and the timestepping algorithm. Features like which set of conservation laws or choice of discretization are encapsulated into the above function calls. In addition to the helper functions, we also define our scheduling primitives. These are the function calls with which events are scheduled in the discrete event simulator.
\begin{itemize}
\item {\sc schedule(t, event)}: Schedule {\sc event} at time {\sc t}.
\item {\sc schedule\_inline(event)}: Execute {\sc event} immediately. 
\end{itemize}

\begin{figure}
\begin{lstlisting}[mathescape=true, basicstyle=\scriptsize]
function update(id, update_forced)
  $\tau \gets$ get_time()
  submesh $\gets$ get_submesh(id)

  if ( $\tau \neq \text{submesh}.\tnext \land \neg \text{update\_forced}$ ) return
  if ( $\tau = \text{submesh}.\tprev$ ) return
    
  advance($\tau$, submesh)
  submesh.$\tnext$ $\gets$ compute_t_next(submesh, $\tau$)
  
  for neigh $\in$ submesh.bdry
    forced_update $\gets$ neigh.$\tprev$ > neigh.$\tsync$
                  $\land$ compute_t_next_bdry(submesh, neigh) $\le \tau$
                  
    schedule( $\tau$, push_flux( neigh.id, id,
                            make_msg(submesh, neigh),
                            forced_update) )
    if ( forced_update ) neigh.$\tsync$ $\gets\tau$
    
  if ( submesh.$\tnext > \tau$ )
    schedule(submesh.$\tprev$, update(id, false))
\end{lstlisting}
\caption{Update function}
\label{fig:update}
\end{figure}

\begin{figure}
\begin{lstlisting}[mathescape=true,basicstyle=\scriptsize]
function push_flux(id, id_from, msg, force_update)
  $\tau \gets$ get_time()
  submesh $\gets$ get_submesh(id)
  neigh $\gets$ get_neigh(submesh, id_from)
  
  accumulate( $\tau$, submesh, neigh )
  neigh.u $\gets$ msg
  neigh.$\tprev$ $\gets \tau$
  if ( neigh.$\tprev =$submesh.$\tprev$ ) neigh.$\tsync$ $\gets \tau$
  update_K_bdry(submesh, neigh)
  
  submesh.$\tnext$ $\gets$ compute_t_next(submesh, $\tau$)
  if ( force_update $\lor$ submesh.$\tnext \le \tau$ )
    schedule_inline(update(id, true))
    
  if ( submesh.$\tnext > \tau$ )
    schedule(submesh.$\tnext$, update(id, false))
\end{lstlisting}
\caption{Push flux function}
\label{fig:push-flux}
\end{figure}

With these helper functions, we now define our two main events: update ($\update$)--shown in Figure~\ref{fig:update}---and push flux ($\pushflux$)--shown in Figure~\ref{fig:push-flux}. The update at time $\tau$ advances the submesh from $\tprev$ to time $\tau$. Once the submesh has updated, it sends values at shared boundaries to its neighbors via a push flux. Due to CFL constraints or local ordering violations, the submesh may require the require the neighbor to update. 
When the neighbor processes the push flux, it integrates the flux to time $\tau$, and then updates the density $u$ and relevant metadata.
If either the CFL condition is no longer satisfied or the neighbor requires an update, the submesh updates inside the body of the push flux.
The simulation schedules initial updates. As each event executes, it schedules subsequent events until every submesh has arrived at the final time. 
The remainder of this section demonstrates that the generated event trace is total variation stable.


%

\begin{theorem}
\label{thm:des-tvd}
Given a minimum timestep size of
\begin{equation}
\Delta t_{\min} < \inf\left\{ \frac{\Delta x_{\min}}{2 \max(K_1(\xi),\,K_2(\xi))}\, : \, \xi \in \range(u_0) \right\},
\label{eq:dt-min}
\end{equation}
where $\Delta x_{\min}$ is the smallest element size, $K_1(\xi)$ is the local Lipschitz constant of $\hat{F}(\cdot,\xi)$ and $K_2(\xi)$ is the local Lipschitz constant of $\hat{F}(\xi,\cdot)$. The event trace generated by the updates and push fluxes with a minimum timestep of $\Delta t_{\min}$ will produce a TVD solution to the scalar conservation law in~\eqref{eq:ivp}.
\end{theorem}

\subsection{Proof of Theorem \ref{thm:des-tvd}}
The proof of the theorem requires demonstrating that the discrete event simulation generates an event trace that satisfies the conditions of Theorem~\ref{thm:tvd-stab}. The main machinery for this proof will be loop invariants. Loop invariants are a formal correctness technique whereby we specify a set of ``correct'' states. By showing that every event maps a correct state onto the set of correct states, once all events have executed, being in the correct state implies satisfying desired criteria. This proof technique has been utilized with great success for systematic design of algorithms for numerical linear algebra as part of the FLAME project~\cite{Bientinesi2011, Myers2018}.
 To demonstrate that the algorithm satisfies the conditions of Theorem~\ref{thm:tvd-stab}, we need invariants to satisfy:
\begin{enumerate}
\item the local ordering principle,
\item the CFL condition,
\item the update rule as expressed in \eqref{eq:update}.
\end{enumerate}
Additionally, we will require two further invariants to show that the computed values are correct, and that duplicated metadata in the boundaries is consistent with the values on the neighboring submesh. To simplify, the notation in these invariants, we assume three submeshes are enumerated as, $S_A$, $S_B$, and $S_C$. The value of $u$ at the cell neighboring another submesh, will be indexed with that neighbor's name. For clarity, we have depicted this naming convention in Figure~\ref{fig:submeshes}.
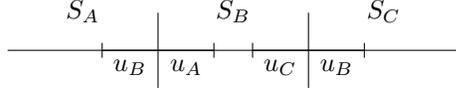
\begin{figure}
\centering
  \begin{tikzpicture}
    \draw (-3,0)-- (-1,0) node[midway, above, label={[distance=1cm]:$S_A$}] {};
    \draw (-1,0)-- ( 1,0) node[midway, above, label={[distance=1cm]:$S_B$}] {};
    \draw ( 1,0)-- ( 3,0) node[midway, above, label={[distance=1cm]:$S_C$}] {};
    \draw (-1,-0.5)--(-1,0.5);
    \draw ( 1,-0.5)--( 1,0.5);
    \draw[|-] (-1.75,0)--(-1,0) node[midway,below] {$u_B$};
    \draw[-|] (-1,0)--(-0.25,0) node[midway,below] {$u_A$};
    \draw[|-] ( 0.25,0)--( 1,0) node[midway,below] {$u_C$};
    \draw[-|] ( 1,0)--( 1.75,0) node[midway,below] {$u_B$};
  \end{tikzpicture}
\caption{Three neighboring submeshes $S_A,\,S_B,$ and $S_C$ with bordering cells named.}
\label{fig:submeshes}
\end{figure}

Local ordering implies that one of two neighboring submeshes must have last updated at the last synchronization time.  To ensure that the event trace will be locally ordered, we define the local ordering invariant for the left interface as
\begin{align}
LO_{\ell}(S_B,\tau) &= (S_B.\tprev = S_B.\ell.\tsync\,\lor\,S_B.\ell.\tprev = S_B.\ell.\tsync) \label{eq:LO1}
\end{align}
We similarly define $LO_r(S_B,\tau)$ for the right interface. The $\tau$ in the arguments of the loop invariants refers to the simulation time at which we evaluate the invariant.

Next, we define the CFL-related invariants, which ensure that our algorithm satisfies the CFL condition~\eqref{eq:cfl-tvd}. Define the invariant for the left boundary cell as
\begin{align}
CFL_{\ell}(S_B, \tau) &= ( S_B.\tnext - S_B.\ell.\tsync) S_B.\ell.K^{\extr} \le 1/2 \label{eq:CFL:ext}\\
& \land\, (S_B.\tnext - S_B.\tprev) S_B.\ell.K^{\intr} \le 1/2, \label{eq:CFL:int}
\end{align}
where $K$ is used to bound the respective $C_j$ and $D_j$ terms. Note that \eqref{eq:CFL:ext} corresponds to interface between submeshes $S_A$ and $S_B$, thus the CFL condition is relative to $S_B.\ell.\tsync$. The other CFL condition \eqref{eq:CFL:int} corresponds to the interior interface of the cell associated with $S_B.u_a$. Since the internal cells step synchronously, the last synchronization time is $S_B.\tprev$, but since $S_B.\ell.K^{\intr}$ depends on $S_B.\ell.u$, this condition must be re-evaluated after each push flux. We  define a CFL invariant for the right interface, $CFL_r(S_B, \tau)$ similarly. We also define an internal CFL invariant, $CFL_{\intr}$, which can be determined by solely examining the internal state of the submesh.
\begin{equation*}
    CFL_{\intr} = \bigwedge_{j=2}^{n_{el}-1} \big( (S_B.\tnext - S_B.\tprev) K_j \le 1/2 \big),
\end{equation*}
where $K_j = \max(-C_j, D_j)$
We then define our CFL invariant as
\begin{equation}
CFL(S_B,\tau) =  (CFL_{\intr}\,\land\,CFL_{\ell}\,\land\,CFL_r). \label{eq:CFL:1}
\end{equation}
The correctness invariant $CR$ checks that all computations are being done correctly. Here, we use $\tilde{S}_B$ to denote the state before executing an event.
\begin{align}
CR&(S_B,\tau) = \big( (S_B.\tprev =\tilde{S}_B.\tprev\,\land\,S_B.u = \tilde{S}_B.u) \label{eq:CR1}\\
&\quad\quad\lor (S_B.\tprev \neq \tilde{S}_B.\tprev \,\land\,S_B\text{ updated according to Equation~\eqref{eq:update}} \big)\label{eq:CR2}\\
&\land \left( S_B.\ell.\Sigma \hat{F} = \int_{S_B.\tprev}^{\max(S_B.\ell.\tprev,S_B.\tprev)} \hat{F}(S_B.\ell.u(\sigma), S_B.u_A) \,\mathrm{d}\sigma \right)\label{eq:CR3}\\
&\land \left( S_B.r.\Sigma \hat{F} = \int_{S_B.\tprev}^{\max(S_B.r.\tprev,S_B.\tprev)} \hat{F}(S_B.u_C, S_B.r.u(\sigma)) \,\mathrm{d}\sigma \right)\label{eq:CR4}\\
& \land\,S_B.\ell.K^{\intr} \ge \max_{\sigma \in (S_B.\tprev, \tau)} D_j(S_B.\ell.u(\sigma), S_B.u_A, S_B.u_{A+1})\label{eq:CR5}\\
& \land\,S_B.\ell.K^{\extr} \ge \max_{\sigma \in (S_B.\ell.\tsync, \tau)} -C_j(S_B.\ell.u(\sigma), S_B.u_A, S_B.u_{A+1})\label{eq:CR6}\\
& \land\,S_B.r.K^{\intr} \ge \max_{\sigma \in (S_B.\tprev, \tau)} -C_j(S_B.u_{C-1}, S_B.u_C, S_B.r.u(\sigma))\label{eq:CR7}\\
& \land\,S_B.r.K^{\extr} \ge \max_{\sigma \in (S_B.r.\tsync, \tau)} D_j(S_B.u_{C-1}, S_B.u_C, S_B.r.u(\sigma)).\label{eq:CR8}
\end{align}
 Equations~\eqref{eq:CR1} and~\eqref{eq:CR2} require the state updates according to the update rule. Equations \eqref{eq:CR3} and~\eqref{eq:CR4} assert flux buffers are correctly integrated at the boundary. The last set of equations \eqref{eq:CR5}--\eqref{eq:CR8} check that the $K$ terms used to enforced the CFL condition correctly bound the Lipschitz constants of the numerical flux. 
 For this invariant to be well-defined, we require that each submesh contain at least two cells. Otherwise, terms like $S_B.u_{A+1}$ would not exist.

The previous three invariants, $LO, CFL,$ and $CR$, have all used information available on submesh $S_B$. However, when duplicating information between submeshes, that information must be consistent. Thus, we define the consistency invariant as
\begin{align}
CI(S_B, S_A, \tau) &= \big( S_B.u_A = S_A.r.u_B \,\land\, S_B.\tprev = S_A.r.\tprev)\label{eq:CI1}\\
&\quad\quad \lor\, \exists\,(\tau,\pushflux(S_A.id, S_B.id, S_B.u_a, \cdot))
\in\pq \big)\label{eq:CI2}\\
&\land\,\big( S_B.\ell.\tsync = S_A.r.\tsync\label{eq:CI3}\\
&\quad \lor\, ( S_B.\ell.\tsync = \tau \,\land\,\exists\,(\tau,\pushflux(S_A.id,S_B.id,\cdot,\cdot))\in \pq\label{eq:CI4}\\
&\quad\lor\, (S_A.r.\tsync=\tau \,\land\,\exists\,(\tau,\pushflux(S_B.id,S_A.id, \cdot, \cdot)) \in \pq \big)\label{eq:CI5}.
\end{align}
Here we use the dot, $\cdot$ to denote any argument. Note that we similarly define consistency for the other interface between $S_B$ and $S_C$.

The last invariant we define is a progress invariant to ensure that the simulation can make progress,
\begin{align}
P(S_B, \tau) &= (S_B.\tprev = t_{end})\\
&\lor  (\exists\,(s,\update(s,S_B,false))\in\pq \, :\, \{s = \,S_B.\tnext \ge \tau > S_B.\tprev\})\label{eq:P2}\\
&\lor  \exists\,(\tau, \pushflux(S_B.id, \cdot, \cdot,\cdot))\in \pq \label{eq:P3} \\
&\lor  \exists\,(\tau, \pushflux(\cdot, S_B.id, \cdot, true))\in\pq. \label{eq:P4}
\end{align}
Combining these five loop invariants, we arrive at
\begin{align*}
I(S_B, S_A, S_C, \tau) &= LO_{\ell}(S_B) \,\land\, LO_r(S_B) \,\land\,CFL(S_B, \tau) \land CR(S_B, \tau)\\
&\land\,CI(S_B,S_A, \tau) \,\land\, CI(S_B,S_C,\tau)\,\land P(S_B,\tau).
\end{align*}

As the aim is to ensure the invariants hold for all submeshes, we will refer to $I(\tau)$ (without the submesh arguments) as
\begin{equation*}
I(\tau) = \bigwedge_{k=1}^{\nsbmsh} I(S_k, S_{k-1}, S_{k+1}, \tau).
\end{equation*}
\begin{proposition}
A discrete event simulation which satisfies the invariant $I$ between all events executed before time $\tau$, will execute a finite number of events at $\tau$ and satisfy $I$ after each event.
\label{prop:1step}
\end{proposition}
Assuming Proposition~\ref{prop:1step} holds, the remainder of the proof of Theorem~\ref{thm:des-tvd} follows by induction.
Satisfying $I$ implies that every update at $\tau$ satisfies the update rule \eqref{eq:update} and the CFL condition \eqref{eq:cfl-tvd}.
Once all the events have executed, remaining in $I$ implies that the event trace is locally ordered and each submesh must have a CFL-satisfying update scheduled for time $\tau+\Delta t_{\min}$ or greater. Lastly, 
Since the solution remains constant from $\tau$ to $\tau + \Delta t_{\min}$, we have established the inductive hypothesis. Arguing inductively, the event trace at $\tend$ is locally ordered, satisfies the CFL condition, and obeys the forward Euler update rule. Therefore, the discrete event simulation will generate a TVD solution.

\subsection{Proof of Proposition \ref{prop:1step}}
The remainder of the proof relies on demonstrating that the simulation state satisfies the invariants after each event is evaluated at time $\tau$. Before we proceed, we prove a progress guarantee.

\begin{proposition}
Given a discrete event simulation that satisfies invariant $I$ up until time $\tau$, $u(\tau)$ will satisfy a maximum principle.
\label{prop:des-max}
\end{proposition}
\begin{proof}
Pick any cell $j$ and consider two update points $t_j^n$ and $t_j^{n+1}$. Let $s_{j-1,j}^{\mu}$ and $s_{j,j+1}^{\eta}$ represent the most recent left and right synchronization times, respectively, and let $u_{j-1}$, $u_j$, $u_{j+1}$ represent the associated cell densities.
Since the CFL invariant is satisfied for all time up until $\tau$,
\begin{align*}
1 &+ \int_{t_j^n}^{t_j^{n+1}} \big[C_j(\sigma) - D_j(\sigma) \big]\,\mathrm{d} \sigma\\
& \ge 1 + (t_j^{n+1} - s_{j-1,j}^{\mu}) \min_{\sigma \in (s_{j-1,j}^{\mu},t_j^{n+1})} C_j(u_{j-1}, u_j, u_{j+1})\\
& \quad\quad - (t_j^{n+1} - s_{j,j+1}^{\eta}) \max_{\sigma \in (s_{j,j+1}^{\eta}, t_j^{n+1})}D_j(u_{j-1}, u_j, u_{j+1})\\
& \ge 0.
\end{align*}
Thus, the event trace satisfies the maximum principle by Theorem~\ref{thm:max}.
\end{proof}

\begin{lemma}[Progress Guarantees]
\label{lem:progress-guarantee}
Given a simulation that satisfies the invariant $I$ up until time $\tau$, if a submesh and its neighbors have updated at time $\tau$, then
\begin{lstlisting}[mathescape=true]
    compute_t_next(submesh, $\tau$) $\ge \tau + \Delta t_{\min}$.
\end{lstlisting}
\end{lemma}
\begin{proof}
Pick any cell and consider the two neighboring cells. Enumerate the densities as $u_{j-1},\,u_j$ and $u_{j+1}$. By Proposition \ref{prop:des-max}, $u_{j-1}, u_j, u_{j+1} \in \range(u_0)$.
By definition of the Lipschitz constant
\begin{equation*}
-C_j(u_{j-1}, u_j, u_{j+1}) \le K_1(u_{j+1})/\Delta x_{\min}.
\end{equation*}
 Therefore,
\begin{align*}
-\Delta t_{\min} C_j(u_{j-1}, u_j, u_{j+1}) \le \frac{K_1(u_{j+1})}{\Delta x_{\min}} \frac{\Delta x_{\min}}{2K_1(u_{j+1})} = \frac{1}{2}
\end{align*}
The proof for the right interface of cell $j$ follows similarly. Thus, taking a timestep of size $\Delta t_{\min}$ satisfies~\eqref{eq:cfl-tvd}.
Since this holds for every cell in the submesh,
\begin{lstlisting}[mathescape=true]
compute_t_next(submesh, $\tau$) $\ge \tau + \Delta t_{\min}$.
\end{lstlisting}

\end{proof}

To see that there are at most a finite number of messages sent at each $\tau$, note that each update will only execute the main body at most once for a given timestep. If multiple updates are scheduled for the same time $\tau$, $S_B.\tprev = \tau$ after the first update, and subsequent updates become no-ops. Furthermore, since push fluxes can only be scheduled when an update has executed at that timestep, we bound the total number of events scheduled at a given timestep by $3 \nsbmsh$.

Completing the proof requires showing that as each event at time $\tau$ is processed, $I$ is not violated. To do so, we will enumerate some states of the submesh and show that execution of any event will not lead to an invariant violation. In order to determine the state of a submesh $S$, we require 6 Boolean variables:
\begin{align*}
b_1 &= (S.\tprev < \tau)  &&b_2 = (S.\tnext > \tau)\\
b_3 &= (S.\ell.\tprev < \tau)  &&b_4 = (S.\ell.\tsync \le S.\ell.\tprev)\\
b_5 &= (S.r.\tprev < \tau)  &&b_6 = (S.r.\tsync \le S.r.\tprev).
\end{align*}

\begin{table}
\caption{All possible states as a submesh $S$ processes various messages. Without loss of generality the first push flux is assumed to arrive from the left.}
\label{tab:firstmsg}
\centering
\begin{tabular}{|c||c|c|c|c|c|c|}
\hline
& $b_1$ & $b_2$ & $b_3$ & $b_4$ & $b_5$ & $b_6$\\ \hline\hline
$q_a$ & True & True & True & True & True & True \\ \hline
$q_b$ & True & False & True & True & True & True \\ \hline
$q_c$ & False & True & True & True & True & True \\ \hline
$q_d$ & False & True $\lor$ False & True & False & True & True\\ \hline
$q_e$ & False & True $\lor$ False & True & True & True & False\\ \hline
$q_f$ & False & True $\lor$ False & True & False & True & False\\ \hline
$q_g$ & True & True & False & True & True & True\\ \hline
$q_h$ & False & True & False & True & True & True \\ \hline
$q_i$ & False & True $\lor$ False & False & True & True & False \\ \hline
$q_j$ & True & True & False & True & False & True \\ \hline
$q_k$ & False & True & False & True & False & True \\ \hline
\end{tabular}
\end{table}

\begin{figure}
\centering
\begin{tikzpicture}[node distance=1.5cm and 3cm, auto]

\node[initial,state,accepting] (A) {$q_a$};

\node[initial,state,accepting] (B) [below= of A, yshift=-1.cm] {$q_b$};

\node[state] (F) [right= of B, yshift=-1.8cm, xshift=-1cm] {$q_f$};
\node[state] (E) [right= of B, yshift=-0.7cm, xshift=-1cm] {$q_e$};
\node[state] (D) [right= of B, yshift=1.5cm, xshift=-1.25cm] {$q_d$};
\node[state] (C) [accepting, right= of B, yshift=2.35cm, xshift=-2cm] {$q_c$};

\node[state,accepting] (H)  [right= of C, yshift=0.2cm] {$q_h$};
\node[state] (I) [right= of E, yshift=1.55cm, xshift=-1.15cm] {$q_i$};

\node[state,accepting] (G) [above= of H,yshift=-1.2cm] {$q_g$};
\node[state,accepting] (J) [right= of G,yshift=-1cm,xshift=-1cm]  {$q_j$};

\node[state,accepting] (K) [right= of H,yshift=-1cm,xshift=-1.25cm] {$q_k$};

\path (A) edge [->, out=30, in=165, above] node {} (G)
          edge [->, out=10, in=150, above] node {} (H)
          edge [->, out=0, in=120,  above] node {} (I)
      (B) edge [->, above, loosely dashdotted]      node {} (C)
          edge [->, above, loosely dashdotted]      node {} (D)
          edge [->, below, loosely dashdotted]      node {} (E)
          edge [->, below, loosely dashdotted]      node {} (F)
          edge [->, out=10, in=225, below]  node {} (H)
          edge [->, out=0, in=195, above] node {} (I)
      (C) edge [->, above]      node {} (H)
      (D) edge [->, above]      node {} (H)
      (E) edge [->, below]      node {} (I)
      (F) edge [->, below]      node {} (I)
      (G) edge [->, above,densely dotted]      node {} (J)
          edge [->, above,densely dotted]      node {} (K)
      (H) edge [->, above,densely dotted]      node {} (K)
      (I) edge [->, above,densely dotted]      node {} (K);
\end{tikzpicture}
\caption{Possible state transitions during a single timestamp. Accepting states are denoted with double circles. The line styles indicate the type of message: the solid line denotes a push flux from the left neighbor, the dotted line denotes a push flux from the right neighbor, the dash-dotted line denotes an update. Without loss of generality we assume that the first push flux processed arrives from the left.}
\label{fig:fsm}
\end{figure}
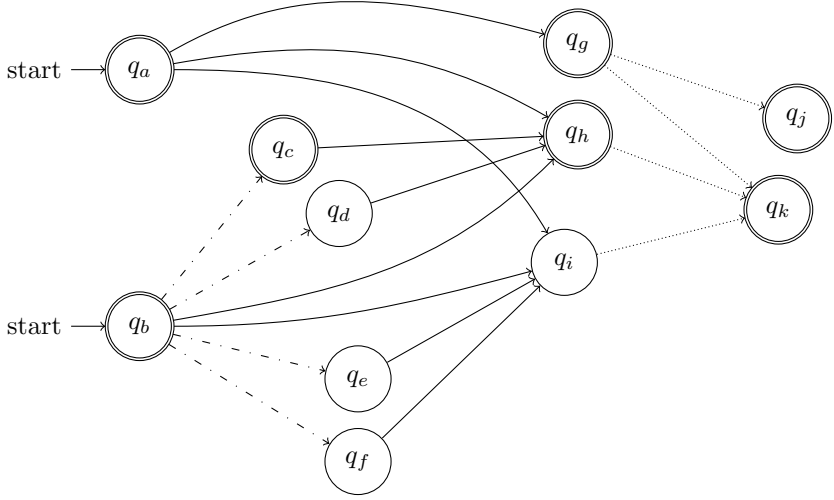
Our aim will be to derive a finite state machine and show that all transitions preserve the loop invariants. To reduce the complexity of the system we assume that if a submesh receives a message, the first message it processes will come from the left.
All states that are attained are shown in Table~\ref{tab:firstmsg} and their transitions are pictorially shown in Figure~\ref{fig:fsm}. We note the non-standard notation depicting multiple arrows coming out of given states. These multiple arrows leaving a state are due to the fact that our binary state representation doesn't fully capture the internal state of any given submesh. In particular, when any given message is processed, the algorithm deterministically computes the appropriate transition based on the values at the cells, e.g. does the submesh need to execute an inlined update. This representation could be expanded to recover the traditional finite state machine representation at significantly greater complexity. 

Before we assess any state transitions, we note that the $CFL$ invariant checks whether or not the proposed $S_B.\tnext$ satisfies the invariant. The function {\sc compute\_t\_next}, which exclusively sets $S_B.\tnext$ is constructed to always satisfy the $CFL$ invariant, i.e. we never propose a timestep that would cause a CFL violation. 
Similarly, the helper functions are implemented to satisfy the correctness invariant $CR$. In both cases, the invariants determine the implementation. Thus, these invariants hold by construction.

We begin by noting that before any submesh has processed any event at time $\tau$, the last events must have been processed at a time strictly less than $\tau$ and synchronization times are only updated when two neighboring submeshes update at the same time, all submeshes must begin in states $q_a$ or $q_b$. These states differ based on whether or not there is an update scheduled at $\tau$.

\subsubsection{Processing Updates}
An update will only be applied to submeshes for which $S_B.\tnext = \tau$ and $S_B.\tprev < \tau$. Otherwise, the update becomes a no-op. Thus, non-trivial update events can only be applied to submeshes in state $q_b$. 
\begin{enumerate}
    \item[$LO$:] To be locally ordered either the submesh or its neighbor must have last updated at the synchronization time. 
    Consider the left side. If $S_B.\ell.\tprev = S_B.\ell.\tsync$, $LO_{\ell}$ continues to hold after the update. Otherwise, the neighbor has updated since the last synchronization, and the current update schedules a push flux that forces the neighbor to update at $\tau$, synchronizing the submeshes at $\tau$.  The update sets the synchronization time $S_B.\ell.\tsync$ to $\tau$ in anticipation of the scheduled synchronization, satisfying the local ordering invariant. The right side follows similarly.
    \item[$P$:] The scheduled push fluxes satisfy~\eqref{eq:P4}.
    \item[$CI$:] Updating a submesh causes the consistency invariant~\eqref{eq:CI1} to fail on the neighbor. However, the submesh sends a push flux. Thus, Equation~\eqref{eq:CI2} holds. If the synchronization times remain unchanged, $S_B.\ell.\tsync=S_A.r.\tsync$ continues to hold. Otherwise, the update schedules a forced push flux, which implies that $S_B.\ell.\tsync=\tau$ and there exists a forced push flux to $S_A$.
\end{enumerate}
 Thus, the submesh satisfies $I$ after processing the update. The submesh transitions to states $q_c-q_f$ according to which sides need to receive a push flux from their neighboring submesh in order to satisfy local ordering or allow scheduling the next non-trivial update. 

\subsubsection{Processing the Push Flux from the Left}
Without loss of generality we assume that the first message comes from the left. Since no push fluxes have been processed, we only consider submeshes in initial states $q_a$ or $q_b$ or states after executing an update $q_c$-$q_f$. 
There are three cases of states to consider: (i) submeshes that have not updated, but update during the push flux, (ii) submeshes that have not updated and do not update during the push flux, and (iii) submeshes that have already updated.
\begin{enumerate}
    \item[$LO$:] For case (i), the processing of the push flux and inlined update implies $S_A$ and $S_B$ are synchronized and $S_B.\ell.\tprev=S_B.\tprev=S_B.\ell.\tsync$. Thus, $LO_{\ell}$ holds. On the right side, $LO_r$ is satisfied using the same logic when updating a submesh. For case (ii), processing the push flux and not updating, implies that $S_A.r.\tprev=S_A.r.\tsync$. Since $S_B.\tprev < \tau$, the consistency invariant implies $S_B.\tprev = S_B.\ell.\tsync$. Since $S_B.\tprev$ and $S_B.r$ are unchanged, $LO_r$ continues to hold. For case (iii), $LO_{\ell}$ holds by the same logic used for case (i), and since $S_B.\tprev$ and $S_B.r$ remain unchanged, $LO_r$ holds.
    \item[$P$:] For case (i), the push fluxes scheduled in the inlined update satisfy \eqref{eq:P4}. For case (ii), since the message was processed and did not lead to an inlined update, $S_B.\tnext > \tau$, and an update has been scheduled satisfying \eqref{eq:P2}. For case (iii), if after the push flux we are able to schedule an update greater than $\tau$, \eqref{eq:P2} is satisfied. Otherwise, the progress guarantee implies that being unable to schedule a future update must arise due to the CFL condition at the right boundary. When the submesh updated it must have sent a forced push flux to the right neighbor, satisfying \eqref{eq:P4}. If that push flux has already been consumed, the right neighbor has updated at $\tau$ and sent a push flux  to the submesh, satisfying \eqref{eq:P3}.
    \item[$CI$:] For case (i) the push flux sent to $S_A$ and the fact that $S_B.\ell.\tsync=\tau$ implies that~\eqref{eq:CI2} and~\eqref{eq:CI4} hold. On the right side, $CI(S_B, S_C, \tau)$ holds following the same logic used for the update.
    For case (ii), 
    $CI(S_B,S_C,\tau)$ continues to hold since it held before the push flux was executed. Considering $CI(S_B,S_A,\tau)$, \eqref{eq:CI1} continues to hold.
    Since $S_B$ did not update, $S_A$ did not schedule a forced update, and $S_A.r.\tsync$ remains unchanged from the state before $S_A$ last updated. Therefore, \eqref{eq:CI3} holds.
    For case (iii), $CI(S_B,S_C,\tau)$ holds following the same logic used for case (ii). On the left interface, $S_B.\ell.\tsync=\tau$ after the push flux. When $S_B$ updated at $\tau$ it sent a message to $S_A$. If this message has not been processed \eqref{eq:CI4} holds. Otherwise, $S_A$ has updated and processed the push flux for $S_B$, implying that  $S_A.r.\tsync=\tau=S_B.\ell.\tsync$ and \eqref{eq:CI3} holds.
\end{enumerate}

It is necessary to check that consuming the push flux does not lead to invariant violations of $P$ or $CI$ on $S_A$.
It can be shown that if $P$ or $CI$ on $S_A$ depend on the existence of the push flux to $S_B$, after executing the push flux, there exists a push flux from $S_B$ to $S_A$, satisfying these invariants.

The transitions following the execution of a push flux from the left neighbor depend on whether or not the submesh updated and whether or not it can make progress on the right side. If the submesh did not update at $\tau$---which is only the case for $q_a$---and it is able to schedule an update at a time greater than $\tau$, this will transition to $q_g$. If the mesh previously updated or is updated during the push flux, the state depends on whether or not the mesh is able to make progress on the right hand side.
States $q_c$ and $q_d$ will transition to $q_h$, and states $q_e$ and $q_f$ will transition to $q_i$. Submeshes which execute an inlined update transition depending on the value of $b_6$. If the submesh has issued a force push flux to its right neighbor to enforce local ordering or enable scheduling a future update, the submesh transitions to $q_i$. Otherwise, the submesh transitions to $q_h$.

\subsubsection{Processing the Push Flux from the Right}
Since by assumption the first message arrives from the left, we only need to consider the right push flux applied to states after processing the push flux from the left neighbor, i.e. states $q_i$--$q_f$.
\begin{enumerate}
    \item[$LO$:] If the push flux did not cause an update, then both sides must satisfy the local ordering constraint. Otherwise, the submesh has synchronized with both neighbors and satisfies the local ordering constraint.
    \item[$P$:] If no update is required, that implies $S_B.\tnext>\tau$. Thus, $P$ is satisfied by \eqref{eq:P2}. Otherwise, if the submesh updated at time $\tau$, it will have already processed both neighboring push fluxes. The progress guarantee implies that the simulation will be able to schedule an update at a time greater than $\tau$, satisfying \eqref{eq:P2}.
    \item[$CI$:] The consistency invariant holds making the same arguments used to show the consistency invariant holds after processing the push flux from the left.
\end{enumerate}
The state transition can be determined by whether or not the submesh updated at $\tau$, if it did not the submesh must have started in state $q_g$ and transitions to $q_j$. Otherwise, the submesh transitions to $q_k$.

\begin{remark} This proof imposed no restriction on scheduling order of events with the same timestamp $\tau$. That is to say that neither update nor push flux event with the same timestamp are required to be executed before the other. 
While any ordering of events will provide a total variation solution, we have not shown that it will generate a unique event trace. 
For our implementation, we enforce deterministic execution by explicitly imposing a deterministic tie breaking mechanism. 
\end{remark}

%% file: sections/implementation_details.tex
\section{Implementation Details}
\label{sec:implementation}

By expressing the algorithm as a discrete event simulation, we can use existing parallelization infrastructure to rapidly and efficiently parallelize the proposed local timestepping algorithm. This section introduces Devastator, the parallel discrete event simulator upon which our implementation is based. Additionally, we outline performance optimizations made to the algorithm and load balancing strategies.

\subsection{Devastator Simulation Framework}
Devastator (publication forthcoming) is the parallel discrete event simulation (PDES) framework we have used to implemented this work. As with other PDES frameworks, Devastator expects the simulation to be modeled as a discrete number of logical processes (i.e. actors) producing and consuming timestamped events. Once the logical processes and their event processing behaviors have been defined, Devastator handles the task of progressing and maintaining consistency of the distributed parallel execution. To do this, Devastator employs the TimeWarp algorithm~\cite{Jefferson1985} with an asynchronous algorithm for bounding global virtual time (GVT). Devastator was chosen for this work for its emphasis on performance in HPC environments as well as its productive C++14 interface.

PDES frameworks generally fall in one of two categories in how they maintain consistency of the distributed state, these are named optimistic and conservative. Consider the following scenario: a CPU $A$ in the simulation wants to execute event $E$ having timestamp $T$. How can the CPU be sure that no other event $E'$ with timestamp $T'$ where $T' < T$ is going to be generated by some other CPU $B$ and sent to $A$? 
Optimistic methods, like TimeWarp, execute events before knowing that it is safe to so. To deal with the inevitable causality violations (discovering $E'$ only after executing $E$), TimeWarp performs a {\em rollback} to revert the logical process's state to before execution of $E$, then executes $E'$, then $E$, and carries on. Once GVT passes $T$, the CPU knows that no such further events $E'$ can occur, and the event $E$ is {\em committed}.

Optimistic execution enjoys the ability to operate with high performance in regimes of low communication, dynamically, simply by virtue that it always assumes no communication is needed before executing the next event. This makes it an excellent choice for domains where absolute bounds on the needed inter-CPU synchronization are far tighter than what is required in the average case. 
 In the case of nonlinear conservation laws, determining whether an event is able to execute without incurring a CFL violation requires determining the domain of dependence for that submesh. Due to pathological examples such as the shockwave for Burgers' equation considered in the introduction, this is an inherently non-local problem. In situ computation of the domain of dependence would greatly increase required communication between submeshes. The locality of this problem can be limited by considering smaller timesteps at the cost of available parallelism. However, in practice we assume that the probability of a remote high-speed wave dramatically increasing $|\Lambda|$ and causing CFL violations is very small. Using optimistic execution, we are able to maintain our local communication stencil without limiting parallelism, and only incur significant overhead when timesteps require refining.

\subsection{Performance related optimizations}
While the above presented push flux and update functions specify a TVD timestepping algorithm, three performance optimizations were necessary to obtain good parallel performance.
\begin{enumerate}
    \item {\em Timestep binning:} While the {\sc compute\_t\_next} allows us to take optimally large timesteps, the local ordering constraint makes this approach suboptimal. 
    Neighboring submeshes with large but unequal timesteps will force one another to update to satisfy the local ordering constraint, lowering the average timestep size.
    We propose binning timesteps to the nearest power of two. Specifically, we require that given a certain timestep size ($\Delta t$), we step to largest multiple of the largest power of two multiplied by $\Delta t_{\min}$ less than the given timestep. 
    \item {\em Reducing unnecessary speculation:} While TimeWarp allows us to speculatively update the submeshes, it is not always wise to do so. If a submesh updates and sends a forced push flux to one neighbor, we know that at that time the neighbor will have to send a push flux back. Any events executed on the submesh before the neighboring push flux arrives must be rolled back. Therefore, whenever scheduling new updates the submesh inspects its state to determine whether it is still waiting for a message from one of its neighbors. If so, it will simply return without scheduling any further events. Once the messages the submesh is waiting on arrive, it will schedule the next events without having to roll back events.
    \item {\em Avoiding small timesteps due to binning:} While timestep binning reduces the number of synchronizations required due to local ordering violations, it introduces another problem due to the fact that timesteps at submesh boundaries are computed relative to the previous synchronization time.
    If a message is delayed the recipient submesh will still be able to make small timesteps---enumerating the bits of the unbinned maximum timestep. Once the delayed message arrives these updates will be rolled back and the work was wasted.
    To remedy this, we introduce a heuristic whereby we examine the ratio between the timestep taken if the submesh were synchronized with its neighbors divided by the computed timestep due to a single neighbor, i.e. the value of {\sc compute\_t\_next\_bdry}. If this ratio is greater than or equal to 2, we force that neighbor to update at the current time. The previous optimization (reducing unnecessary speculation) then causes the submesh to wait until the neighbor's push flux has been processed.
\end{enumerate}

\subsection{Performance Modeling and Load Balancing}
\label{sec:load-balancing}
To understand the performance results as well as load balance the problem, we estimate the work for a given problem as follows. At a given simulation time $\tau$, the timestep $dt$ taken by cell $j$ can be approximated as
\begin{equation*}
dt_j(\tau) = \Delta t_{\min} 2^{\left\lfloor \log_2 \frac{ \Delta x_j}{2 |\Lambda(\tau, x_j)| \Delta t_{\min}}\right\rfloor},
\end{equation*}
where $\Delta t_{\min}$ is the smallest timestep, $\Delta x_j$ is the cell size, and $|\Lambda|$ corresponds to the wave speed at time $\tau$ in the midpoint of the cell. The $\log_2$  appropriately bins the timesteps.

The mesh partitioning problem follows a 2-phase process of aggregating cells into submeshes and assigning the submeshes to ranks. Cells within each submesh step synchronously enabling efficient utilization of CPU architectures. 
For the generation of the mapping of cells to submeshes, which we denote by $\pi$, we assume no prior knowledge on $|\Lambda|$, i.e. $|\Lambda| \equiv 1$, and partition solely based on variation in cell sizes. The cell work $w_j$ is the inverse of the smallest timestep taken by any cell assigned to the same submesh.
Since the partitioner balances work across submeshes, but the work depends on the partitioning, there exists a circular dependency.
Hence, we generate $\pi$ using an iterative procedure. Given weights $\{w_j\}$, we create a partition $\pi$, ensuring each submesh has the same amount of work.
 With the new partition, update the weights $\{w_j\}$. Repeat this process until some terminating condition is satisfied. In this paper, we stop iterating after 100 iterations. At the end of the iterative procedure, we return the partition with the least overworked submesh, i.e. if $\pi_{\ell}$ is the $\ell$-th iteration of the submesh partitioner, 
\begin{equation*}
\pi = \argmin_{\{\pi_{\ell}\}} \max_{0 < i < \nsbmsh} \sum_{j=1}^{n_{el}} \chi_{\{\pi_{\ell}(j) = i\}}(j) w_{j}
\end{equation*}
where $\chi_A$ denotes an indicator function over set $A$. The objective of the first partitioning phase is to specify a problem that can be efficiently executed by the parallel discrete event simulator. The proposed scheme generates submeshes such that the number of cells updated is approximately equal across all submeshes. However, due to the dependence of work on partitioning, we must also pay attention to the total work. A proposed partitioning may be perfectly load balanced but require more work and therefore be less preferable than a partitioning which is imbalanced but more work optimal, i.e. for which the discrepancy between $w_j$ and $dt_j^{-1}$ is smaller.

The second partitioning phase assigns submeshes to ranks. The objective of this phase is to minimize the runtime of the simulation. Let $\rho$ map a submesh to its assigned rank. Since the most overworked rank will determine the rate at which GVT is advanced, we estimate the wall-clock time as
\begin{equation}
T = \int_0^{\tend} \max_{0 \le k < n_{ranks}} \sum_{j=1}^{n_{el}}  w_j(\tau) \chi_{\{(\rho \circ \pi)(j) = k\}}(j) \,\mathrm{d} \tau.
\label{eq:lb}
\end{equation}
We remark that for performance results presented in Section~\ref{sec:performance-results}, we consider analytic solutions, thereby yielding a good approximation to $|\Lambda|$ and hence $w_j$.
By using a quadrature rule to approximate the integral, we formulate \eqref{eq:lb} as a mixed integer programming problem, which we solve using Gurobi~\cite{Gurobi}.  

Lastly, we derive upper limits on the achievable speed-up as the ratio of work (i.e. number of cell updates) executed by a standard synchronous timestepping implementation using an MPI runtime divided by the local timestepping Devastator implementation. We calculate the total work for the Devastator runtime as
\begin{equation*}
W^{th}_{deva} = \sum_{j=1}^{n_{el}} \int_0^{\tend} w_{j}(\tau) \, \mathrm{d} \tau.
\end{equation*}
Note that this estimate is a conservative work estimate for the Devastator implementation as it doesn't account for factors such as taking intermediate timesteps between finer and coarser timesteps as well as extra timesteps due to forced updates and rollbacks. For an MPI-based implementation, which steps with uniform timesteps, we compute the work $W^{th}_{MPI}$ as the number of timesteps times the number of cells.
We bound the largest theoretical speed-up by
\begin{equation*}
    S^{th} = \frac{ W^{th}_{MPI}}{W^{th}_{deva}}.
    \label{eq:Sth}
\end{equation*}
This will provide a baseline to assess how efficient the implementation is.

%% file: sections/results.tex
\section{Results}
\label{sec:results}
\begin{figure}
  \centering
  \subfloat[Uniform]{\includegraphics{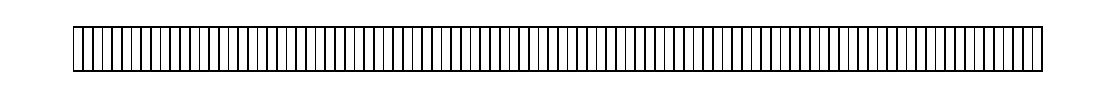}}\\
  \subfloat[Polynomial]{\includegraphics{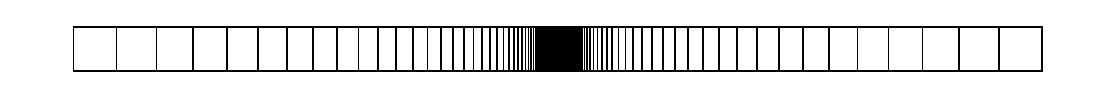}}\\
  \caption{Meshes used for numerical experiments}
  \label{fig:meshes}
\end{figure}

In this section we present results for the one dimensional shallow water equations. Since the timestepping method is first order, we only consider first order finite volume schemes.
To demonstrate the robustness of the timestepping method for different types of meshes, we consider two meshes: a uniform mesh and a polynomial mesh. These meshes are generated by warping uniformly distributed nodes along $(-1,1)$ onto the interval $(-1,1)$. The base case is the uniform mesh for which the warp function is the identity, i.e. $w(x) = x$. The polynomial mesh refines the mesh around the origin. This type of mesh is common for finite element applications where local refinement is required to resolve flow around fine features. The warp function for this mesh is given as
\begin{equation*}
  w(x) = \frac{1}{1/3 + \varepsilon} \left( \frac{ x^3}{3} + \varepsilon x \right),
\end{equation*}
where we set $\varepsilon = 0.02$ in order the bound the ratio of largest to smallest cells. For clarity the meshes used are depicted in Figure~\ref{fig:meshes}. To illustrate the behavior of the timestepping algorithm, we consider meshes with 100 cells and 20 submeshes in the next two sections.  In practice, the number of cells per submesh needs to be significantly larger to amortize runtime overheads with useful work. Section~\ref{sec:performance-results} showcases performance results with meshes consisting of 500,000 cells.

\subsection{Shallow Water Equations}
To show that the local timestepping scheme is robust for non-linear systems, we consider the shallow water equations,
\begin{equation*}
  \begin{cases}
    \partial_t h + \partial_x q_x = 0 &\\
    \partial_t q_x + \partial_x ( hu^2 + gh^2/2) = -gh\partial_x z&\\
  \end{cases}
\end{equation*}
where $q_x = hu$, $z$ is the bathymetry, and $g=1$. We consider here the dam break Riemann problem, for which,
\begin{align*}
  h_0(x) &= \begin{cases}
    1 & \text{if } x < 0,\\
    1/16.1 & \text{if } x > 0,
  \end{cases}\\
  q_{x,0}(x) &= 0,\\
  z &= 0.
\end{align*}
For the shallow water equations, the maximum advection speed, $|\Lambda|$ is  $\sqrt{gh} + |u|$. The initial conditions have been chosen to allow a 4-to-1 timestepping ratio between downstream and upstream of the dam break for the uniform mesh. The second test case we consider is the analytic problem of Carrier and Greenspan~\cite{Carrier1958}. This problem considers water moving up and down a shoreline with uniform slope in a periodic manner. We follow the set-up outlined in~\cite{Bokhove2005} with a phase shift of $\varphi=-\pi$.

For the numerical discretization, we use a local Lax-Friedrichs flux along with the first-order local hydrostatic reconstruction proposed in~\cite{Audusse2004}. The space time plots are shown in Figure~\ref{fig:swe:spacetime}.  We note that for the shallow water flow, the theoretical guarantees no longer hold. In fact, the timestepping region between the two waves emanating from the dam break problem in Figure~\ref{fig:swe:spacetime:uniform:dambreak} requires a finer timestep than by the initial conditions, and thereby violates the progress guarantee, Lemma~\ref{lem:progress-guarantee}. 

Looking at Figure~\ref{fig:swe:spacetime:uniform:dambreak}, we see that submeshes refine their timesteps once the waves from the Riemann problem locally arrive. On the polynomial mesh---shown in Figure~\ref{fig:swe:spacetime:polynomial:dambreak}---around $t=0.14$ the entire mesh begins stepping at the finest frequency. This is due to the fact that cells in the interval $(0.146, 1)$ belong to a single submesh. Since these cells take larger timesteps the submesh partitioner places more cells into this submesh. 

For the Carrier-Greenspan problem on the uniform mesh, shown in Figure~\ref{fig:swe:spacetime:uniform:cg}, we see water moving in and out of the domain. Analytically, the water front never goes past $x=0.25$, and the simulated event trace reflects this as no submesh updates occur for in regions for which $x>0.3$.  For the polynomial mesh in Figure~\ref{fig:swe:spacetime:polynomial:cg}, we again begin stepping everywhere due to the submesh graph partitioning. For both of these cases, we note hysteretic effects in drying regions. The final simulation time of $2 \pi$ contains two periods of the Carrier-Greenspan solution. At time $\pi$,  ideally, only submeshes which contain cells located at $x<-0.25$ would update. Rather we see a slow ``draining'' of mass as submeshes try to coarsen their timesteps, resulting in updates occurring in regions of the mesh that are dry in the analytic solution. 

To conclude, the proposed timestepping method remains stable for simulations with significant temporal variations in $\Lambda$ for the shallow water equations. While we fail to satisfy the assumptions for the proof of correctness and even observe a violation of the progress guarantee, the algorithm is nevertheless able to stably compute the correct solution. 

\begin{figure}
  \centering
  \subfloat[Uniform-Dam break\label{fig:swe:spacetime:uniform:dambreak}]{\includegraphics{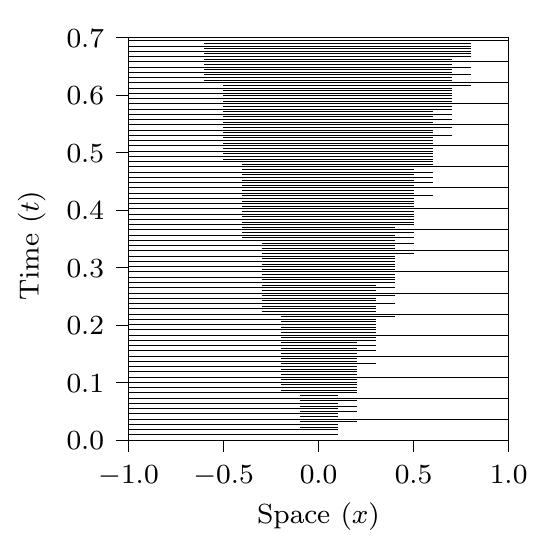}}
  \hfill
  \subfloat[Uniform-Carrier-Greenspan\label{fig:swe:spacetime:uniform:cg}]{\includegraphics{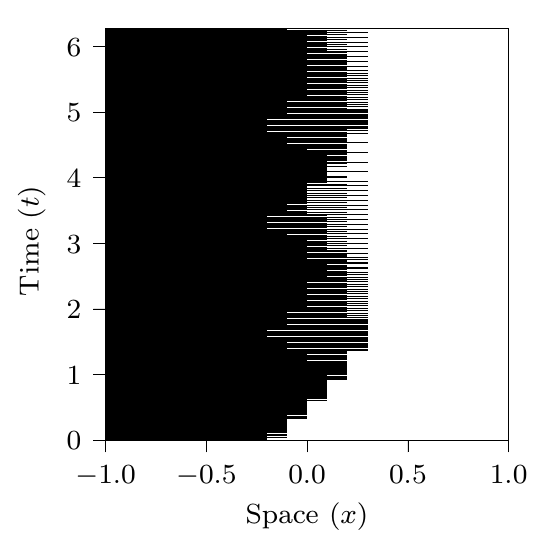}}
  \\
  \subfloat[Polynomial-Dam break\label{fig:swe:spacetime:polynomial:dambreak}]{\includegraphics{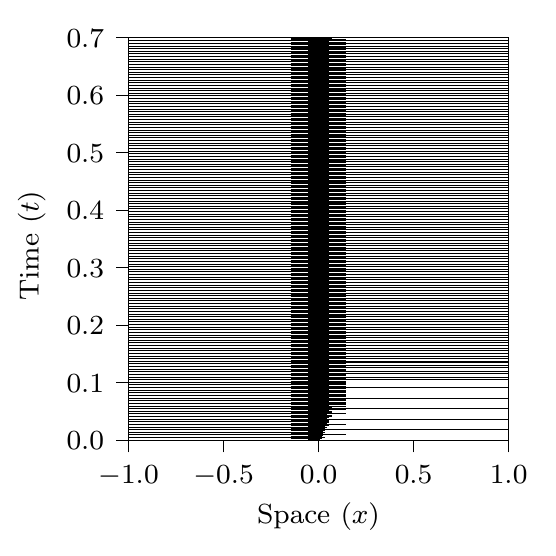}}
  \hfill
  \subfloat[Polynomial-Carrier-Greenspan\label{fig:swe:spacetime:polynomial:cg}]{\includegraphics{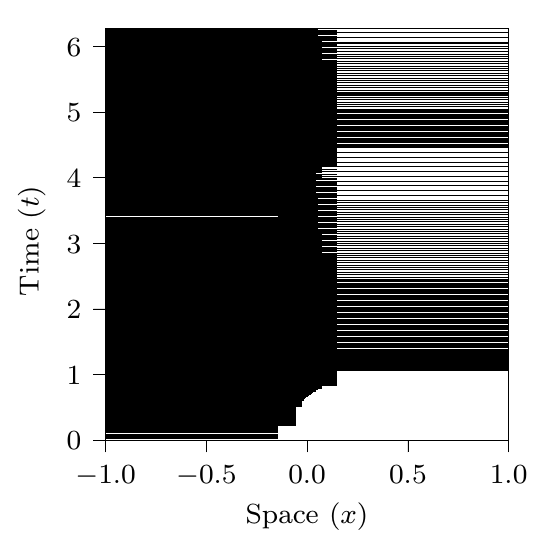}}
  \caption{Space-time plots for the shallow water equations with various meshes and problem configurations.}
  \label{fig:swe:spacetime}
\end{figure}

\subsection{Performance Comparison}
\label{sec:performance-results}

In this section, we compare the performance of our local timestepping implementation to a flat MPI implementation. For the MPI implementation, we use non-blocking point to point messaging and hide message latencies with internal work. One key detail is that the MPI implementation does not compute the CFL condition, but rather uses a fixed timestep. Since dynamically updating a CFL condition would require an all-reduce at each timestep, the communication overhead makes a synchronous adaptive CFL condition non-viable for large-scale simulation. The timestep $\Delta t_{\min}$ is set to  $0.5 (\Lambda/\Delta_x)_{\max}$ based on an analytic expression of the solution. For the devastator runs, we halve the MPI timestep to demonstrate that suboptimal $\Delta t_{\min}$ does not degrade performance.
For the performance comparison, we consider the uniform and polynomial mesh with 500,000 cells on one Skylake node with 48 cores on TACC's Stampede2. We partition the mesh into 48 (one per core) uniform submeshes for the MPI implementation, and 288 (six per core) submeshes for the Devastator implementation.

To analyze the impact of the dynamic CFL condition versus  local timestepping due to mesh refinement, we consider two problems for the shallow water equations. Firstly, we consider a lake at rest problem, where
\begin{align*}
h_0(x) &= 1,\\
q_{x,0}(x) &= 0,\\
z &= 0.
\end{align*}
For this problem, differences in the CFL condition are solely a function of differences in cell size, $\Delta x$. The second problem we consider is the Carrier-Greenspan solution, outlined in the previous section. 
For the lake at rest problem, the work in each submesh is approximately the same. Thus, we assign submeshes $[6\cdot k, 6 \cdot(k+1))$ to rank $k$. For the Carrier-Greenspan problem, we load balance using the heuristic outlined in Section~\ref{sec:load-balancing}. The performance data is shown in Table~\ref{tab:performance-data}. For the static uniform Lake at Rest problem, we include a Devastator configuration---labeled as (no CFL)---which skips computing the local CFL in order to provide a more direct comparison with the MPI version.  Each configuration is run 10 times with the average time elapsed, $\bar{T}$ reported in seconds. The standard deviation over the mean $\sigma_T/\bar{T}$ is given in percentages.


\begin{table}
{\footnotesize
\caption{Execution times for the shallow water equations on Stampede2.}
\label{tab:performance-data}
  \centering
\include{images/performance_data}
}
\end{table}

\begin{table}
{\footnotesize
\caption{Theoretical versus observed speed-ups.}
\label{tab:speed-ups}
  \centering
\include{images/speedups}
}
\end{table}
The speed-up achieved via local timestepping is attained through work reduction. 
To assess the quality of our implementation, we contextualize execution times with configuration-dependent metrics to determine what fraction of unnecessary work was skipped. 
Recall the theoretical speed-up \eqref{eq:Sth}, which estimates the maximum achievable speed-up based on analytic values of $\Lambda$ and $\Delta x$. We compare this theoretical estimate against a speed-up based on the number of cells updated.
 Let $W^{obs}_{MPI}$ and $W^{obs}_{deva}$ be the number of cells updated during the simulation with respective runtimes, and define the work-based speed-up as $S^{work}$ the ratio of observed updates, i.e.
\begin{equation*}
S^{work} = \frac{W^{obs}_{MPI}}{W^{obs}_{deva}}.
\end{equation*}
We introduce the observed speed-up as the ratio of run times,
\begin{equation*}
S^{obs} = \frac{ \bar{T}_{MPI} }{\bar{T}_{deva}},
\end{equation*} 
where $\bar{T}_*$ corresponds to the mean observed execution time for either runtime. 
Table \ref{tab:speed-ups} presents the three speed-up values for each configuration. The discrepancies between $S^{th}$ and $S^{work}$ are less than 3\% with the exception of the Carrier-Greenspan polynomial configuration where $S^{th}$ is 17\% larger than $S^{work}$. We suspect these discrepancies result due to difficulties in water draining out of dry regions akin to the updates seen in Figures~\ref{fig:swe:spacetime:uniform:cg} and \ref{fig:swe:spacetime:polynomial:cg}.

 The observed speed-up ranges from $\minSObsOverWork\%$ to $\maxSObsOverWork\%$ of the work-based speed-up, $S^{work}$. 
 Since there is no theoretical benefit from using adaptive, local timestepping in the lake at rest on a uniform mesh configuration, this case highlights the overhead associated with using an adaptive, local timestepping scheme versus a static synchronous timestepping scheme. If we skip the CFL computation and use the analytic value of $K^{\intr}$ (see the first row of Table~\ref{tab:performance-data}), the observed speed-up is  $S^{obs}=\SobsNoCFL$. 

\begin{figure}
\centering
\includegraphics{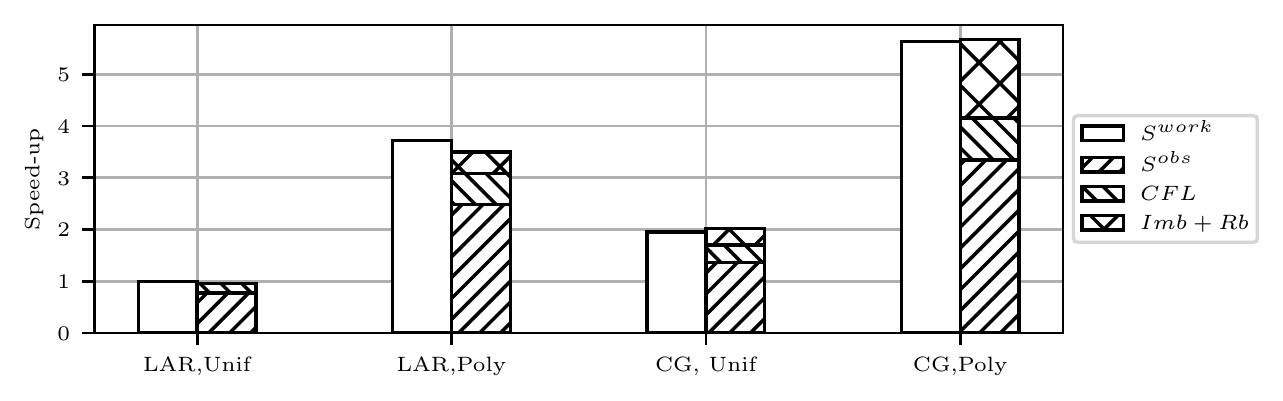}\vspace{-5mm}
\caption{Comparison of speed-up due to work  reduction $S^{work}$ against observed speed-up $S^{obs}$ and the impact of avoiding CFL computations $(CFL)$ and fixing load imbalance and rollbacks $(Imb + Rb)$ for the lake at rest (LAR) and Carrier-Greenspan (CG) problem on the uniform (Unif) and polynomial meshes (Poly).}
\label{fig:imbalance-bars}
\end{figure}

To distinguish between algorithmic overhead of computing the CFL condition and other performance issues, we estimate the cost of the CFL condition by multiplying the MPI execution times by the ratio of the execution times between the CFL and the no-CFL configurations of the lake at rest problem on the uniform mesh using the Devastator runtime. 
Furthermore, the impact of load imbalance and rollbacks is approximated by considering the average imbalance during the simulation. The imbalance is computed as the ratio of cells updated (both rolled back and committed) on the most overworked rank divided by the average number of committed cell updates across all ranks. Over a sufficiently small time interval this imbalance ratio approximates the improvement obtained through perfect load balancing. We estimate the impact of load imbalance by considering average imbalance over 100 evenly sized time intervals.

The cumulative performance impacts of the CFL computation and load imbalance are shown in Figure~\ref{fig:imbalance-bars}. This graph compares the speed-up due to work reduction, $S^{work}$ against the observed speed-up $S^{obs}$ and illustrates how the CFL computation and load imbalance account for the discrepancy. It is important to note that a given part of the stacked bar graph multiplies factors below it, e.g. the top of the $Imb+Rb$ bar corresponds to the expected speed-up if the simulation were perfectly balanced {\em and} there was no cost associated with the CFL condition. 
With these two factors, we are able to account for the discrepancy in $S^{work}$ and $S^{obs}$ speed-ups within $\maxSpeedupErr\%$.

We also observed that rollback had a limited impact on performance for our experimental configurations. For the Carrier-Greenspan problem, rollback minimally impacted the imbalance of the simulation with the imbalance metric increasing by no more than \imbCGdtRb{} for either mesh. For the polynomial lake at rest problem, the imbalance of committed updates was only \polynomialImb . However, once rollbacks are taken into consideration the imbalance went up to \polynomialImbwRb . Rollbacks are observed near regions of large CFL variation. For the Carrier-Greenspan problem, the lack of local clustering in the mapping of submeshes to ranks distributed rollbacks across ranks, and the impact of rollbacks on imbalance is relatively limited. However, for the polynomial lake at rest problem, each rank is assigned a contiguous chunk of submeshes, and rollback occurs on a few ranks.

Another cause of the performance degradation due to imbalance may be limitations of static load balancing. Discrepancies between the performance model  and observed work done may lead to load imbalances, which we are not accounting for. Furthermore, the lower bound for the imbalance of the Carrier-Greenspan problem at the end of the Gurobi partitioning was 20\%, suggesting there are underlying limits to how well the partitioner is able to statically load balance the problem. 

Lastly, we note that while the overhead associated with computing the CFL condition may seem high, it is worth noting that the MPI implementation is not provably stable. For the results presented here, we are able to take optimal timestep sizes, because we are able to analytically determine $(\Lambda / \delta x)_{\max}$. In practice, this value is unobtainable, and the selected $\Delta t$ would lead to sub-optimal timestepping, i.e. taking timesteps smaller than necessary. The local timestepping algorithm can set an arbitrarily small $\Delta t_{\min}$,  and then take appropriately sized timesteps with step-size reductions taken as needed to guarantee stability.

%% file: images/performance_data.tex

\begin{tabular}{|c|c|c|c|c|c|c|c|c|} \hline
Problem & Mesh & Runtime & $\bar{T}$ & $\sigma_T/\bar{T}$ & $\#$ of Updates & $\#$ of Rollbacks\\\hline
Lake at Rest & Uniform & Devastator (no CFL) &  139.4 & 0.8\% & $252.5\cdot 10^9$ & 0 \\\hline
Lake at Rest & Uniform & Devastator &  173.2 & 0.8\% & $252.5\cdot 10^9$ & 0 \\\hline
Lake at Rest & Uniform & MPI &  133.2 & 0.1\% & $252.5\cdot 10^9$ & -- \\\hline
Lake at Rest & Polynomial & Devastator &  192.6 & 3.3\% & $239.7\cdot 10^9$ & $  3.5\cdot 10^9$ \\\hline
Lake at Rest & Polynomial & MPI &  478.4 & 2.0\% & $892.2\cdot 10^9$ & -- \\\hline
Carrier-Greenspan & Uniform & Devastator &  729.4 & 1.1\% & $856.2\cdot 10^9$ & $ 12.0\cdot 10^9$ \\\hline
Carrier-Greenspan & Uniform & MPI &  995.7 & 0.7\% & $1674.5\cdot 10^9$ & -- \\\hline
Carrier-Greenspan & Polynomial & Devastator & 3212.4 & 0.5\% & $3209.6\cdot 10^9$ & $113.3\cdot 10^9$ \\\hline
Carrier-Greenspan & Polynomial & MPI & 10732.0 & 0.1\% & $18070.1\cdot 10^9$ & -- \\\hline
\end{tabular}

%% file: images/speedups.tex

\begin{tabular}{|c|c|c|c|c|} \hline
Problem & Mesh & $S^{th}$ & $S^{work}$ & $S^{obs}$ \\\hline
Lake at Rest & Uniform & 1.00 & 1.00 & 0.77 \\\hline
Lake at Rest & Polynomial & 3.73 & 3.72 & 2.48 \\\hline
Carrier-Greenspan & Uniform & 2.02 & 1.96 & 1.37 \\\hline
Carrier-Greenspan & Polynomial & 6.59 & 5.63 & 3.34 \\\hline
\end{tabular}

%% file: sections/conclusion.tex
\section{Conclusion}

This work presented an adaptive local timestepping algorithm for conservation laws. Loop invariants were used to derive a proof of formal correctness, ensuring that the algorithm is total variation stable even when considering dynamic changes in local wave speeds. Furthermore, the algorithm was parallelized using a speculative parallel discrete event simulator. Results indicate that the parallelization recovers a majority of the expected speed-up, when accounting for the cost of the CFL condition.
As part of our future work, we will examine high-order multi-dimensional timestepping strategies  in the spirit of~\cite{Constantinescu2007}. The forward Euler step forms the basis of these methods methods. Furthermore, we aim to apply this method to physically relevant coastal modeling problems to better quantify benefit of adaptive local timestepping to problems of interest.

